\documentclass{amsart}

\usepackage[all]{xy}

\theoremstyle{plain}
\newtheorem{theorem}[subsection]{Theorem}
\newtheorem{proposition}[subsection]{Proposition}
\newtheorem{lemma}[subsection]{Lemma}
\newtheorem{corollary}[subsection]{Corollary}

\theoremstyle{definition}

\newtheorem{conjecture}[subsection]{Conjecture}

\theoremstyle{remark}
\newtheorem{remark}[subsection]{Remark}

\numberwithin{equation}{subsection}

\author{Tabes Bridges}
\author{Rankeya Datta}
\author{Joseph Eddy}
\author{Michael Newman}
\author{John Yu}

\thanks{Results in this paper were obtained during an REU at
Columbia University led by A.J. de Jong in the summer of 2012.}

\email{Tabes Bridges mtb2123@columbia.edu}
\email{Rankeya Datta rd2443@columbia.edu}
\email{Joseph Eddy jle2129@columbia.edu}
\email{Michael Newman mgnewman@umich.edu}
\email{John Yu jy2433@columbia.edu}
\begin{document}

\title{Free and very free morphisms into a Fermat hypersurface}

\maketitle

\section{Introduction}

\noindent
Any smooth projective Fano variety in characteristic zero is rationally
connected and hence contains a very free rational curve. In positive
characteristic a smooth projective Fano variety is rationally chain
connected. However, it is not known whether such varieties are separably
rationally connected, or equivalently, whether they have a very free
rational curve. This is an open question even for nonsingular
Fano hypersurfaces. See \cite{kollar}, as well as \cite{debarre}.

\medskip\noindent
Following \cite{mingmin}, we consider the degree $5$
Fermat hypersurface
$$
X \quad:\quad X_0^5 + X_1^5 + X_2^5 + X_3^5 + X_4^5 + X_5^5 = 0
$$
in $\mathbb{P}^5$ over an algebraically closed field $k$ of characteristic $2$.
Note that $X$ is a nonsingular projective Fano variety.

\begin{theorem}
\label{theorem-main}
Any free rational curve $\varphi : \mathbb{P}^1 \to X$ has degree $\geq 8$
and there exists a free rational curve of degree $8$. Any very free rational
curve $\varphi : \mathbb{P}^1 \to X$ has degree $\geq 9$ and there exists a
very free rational curve of degree $9$.
\end{theorem}

\noindent
This result, although perhaps expected, is interesting for several
reasons. First, it is known that $X$ is unirational,
see \cite[Page 52]{debarre} (the corresponding rational map
$\mathbb{P}^4 \dasharrow X$ is inseparable). Second, in
\cite{beauville}, it is shown that every nonsingular hyperplane section of $X$
is isomorphic to a Fermat hypersurface of dimension 3 and this property
characterizes Fermat hypersurfaces among all hypersurfaces of degree 5
in characteristic $2$. We believe that these facts single out the Fermat as a
likely candidate for a counter example to the conjecture below; instead
our theorem shows that they are evidence for it.

\begin{conjecture}
\label{conjecture-fano-src}
Nonsingular Fano hypersurfaces have very free rational curves.
\end{conjecture}

\noindent
The paper \cite{zhu} disusses this question more broadly.
A little bit about the method of proof. In
Section \ref{section-setup}
we translate the geometric question into an algebraic question which is
computationally more accessible. In
Sections \ref{splittingrelation}, \ref{section-numerology},
and \ref{section-degree-4-5} we exclude low degree solutions by theoretical
methods. Finally, in
Sections \ref{section-degree-8} and \ref{section-degree-9}
we explicitly describe some curves which
are free and very free in degrees 8 and 9 respectively.

\section{The overall setup}
\label{section-setup}

\noindent
In the rest of this paper $k$ will be an algebraically closed field of
characteristic $2$ and $X$ will be the Fermat hypersurface of
degree $5$ over $k$. Let $\varphi : \mathbb{P}^1 \to X$
be a nonconstant morphism. We will repeatedly use that every vector
bundle on $\mathbb{P}^1$ is a direct sum of line bundles, see
\cite{grothendieck}. Thus we can choose a splitting
$$
\varphi^*T_X =
\mathcal{O}_{\mathbb{P}^1}(a_1) \oplus
\mathcal{O}_{\mathbb{P}^1}(a_2) \oplus
\mathcal{O}_{\mathbb{P}^1}(a_3) \oplus
\mathcal{O}_{\mathbb{P}^1}(a_4).
$$
Recall that $\varphi$ is said to be a {\it free curve} on $X$ if
$a_i \geq 0$ and $\varphi$ is said to be {\it very free} if $a_i > 0$.
Consider the following commutative diagram
\begin{equation}
\label{equation-diagram}
\vcenter{
\xymatrix{
& 0 \ar[d] & 0 \ar[d] \\
& \mathcal{O}_X \ar@{=}[r] \ar[d] & \mathcal{O}_X \ar[d] \\
0 \ar[r] &
E_X \ar[d] \ar[r] &
\mathcal{O}_X(1)^{\oplus 6} \ar[d] \ar[r] &
\mathcal{O}_X(5) \ar@{=}[d] \ar[r] &
0 \\
0 \ar[r] &
T_X \ar[r] \ar[d] &
T_{\mathbb{P}^5}|_X \ar[r] \ar[d] &
N_{X/\mathbb{P}^5} \ar[r] &
0 \\
& 0 & 0
}
}
\end{equation}
with exact rows and columns as indicated. We will call $E_X$ the
{\it extended tangent bundle of $X$}. The left vertical exact sequence
determines a short exact sequence
$$
0 \to \mathcal{O}_{\mathbb{P}^1} \to \varphi^*E_X \to \varphi^*T_X \to 0.
$$
The splitting type of $\varphi^*E_X$ will consistently be denoted
$(f_1, f_2, f_3, f_4, f_5)$ in this paper. Since
$\text{Hom}_{\mathbb{P}^1}(\mathcal{O}_{\mathbb{P}^1}(f),
\mathcal{O}_{\mathbb{P}^1}(a)) = 0$
if $f > a$ we conclude that
\begin{enumerate}
\item If $f_i \geq 0$ for all $i$, then $\varphi$ is free.
\item If $f_i > 0$ for all $i$, then $\varphi$ is very free.
\end{enumerate}
For the converse, note that the map
$\mathcal{O}_{\mathbb{P}^1} \to \varphi^*E_X$ has image contained
in the direct sum of the summands with $f_i \geq 0$. Hence, if $f_i < 0$
for some $i$, then $\varphi$ is not free. Finally, suppose that
$f_i \geq 0$ for all $i$. If there are at least two $f_i$ equal to $0$,
then we see that $\varphi$ is free but not very free. We conclude that
\begin{enumerate}
\item[(3)] If $\varphi$ is free, then $f_i \geq 0$ for all $i$.
\item[(4)] If $\varphi$ is very free, then either (a) $f_i > 0$ for all $i$,
or (b) exactly one $f_i = 0$ and all others $> 0$.
\end{enumerate}
We do not know if (4)(b) occurs.

\medskip\noindent
{\bf Translation into algebra.}
Here we work over the graded $k$-algebra $R = k[S,T]$. As usual, we let
$R(e)$ be the graded free $R$-module whose underlying module is $R$ with
grading given by $R(e)_n = R_{e + n}$. A {\it graded free $R$-module}
will be any graded $R$-module isomorphic to a finite direct sum of $R(e)$'s.
Such a module $M$ has a {\it splitting type}, namely the sequence of integers
$u_1, \ldots, u_r$ such that $M \cong R(u_1) \oplus \ldots \oplus R(u_r)$.

\medskip\noindent
We will think of a degree $d$ morphism
$\varphi: \mathbb{P}^1 \to \mathbb{P}^5$ as a $6$-tuple
$(G_0, \ldots, G_5)$ of homogeneous elements in $R$ of degree $d$
with no common factors. Then $\varphi$ is a morphism into $X$ if and only if
$G_0^5 + \ldots + G_5^5 = 0$. In this situation we define two graded
$R$-modules. The first is called the {\it pullback of the cotangent bundle}
$$
\Omega_X(\varphi) =
\text{Ker}(\tilde\varphi : R^{\oplus 6}(-d) \longrightarrow R)
$$
where the map $\tilde\varphi$ is given by
$(A_0, \ldots, A_5) \mapsto \sum A_iG_i$. The second is called the
{\it the pullback of the extended tangent bundle}
$$
E_X(\varphi) = \text{Ker}(R^{\oplus 6}(d) \longrightarrow R(5d))
$$
where the map is given by $(A_0, \ldots, A_5) \mapsto \sum A_iG_i^4$.
Since the kernel of a map of graded free $R$-modules is a graded free
$R$-module, both $\Omega_X(\varphi)$ and $E_X(\varphi)$
are themselves graded free $R$-modules of rank $5$.

\begin{lemma}
\label{lemma-splitting-type}
The splitting type of $\varphi^*E_X$ is equal to the splitting type
of the $R$-module $E_X(\varphi)$.
\end{lemma}

\begin{proof}
Recall that $\mathbb{P}^1 = \text{Proj}(R)$. Thus, a finitely generated
graded $R$-module corresponds to a coherent sheaf on $\mathbb{P}^1$, see
\cite[Proposition 5.11]{Hartshorne}.
Under this correspondence, the module $R(e)$ corresponds to
$\mathcal{O}_{\mathbb{P}^1}(e)$. The lemma follows if we show that
$\varphi^*E_X$ is the coherent sheaf associated to $E_X(\varphi)$.
Diagram (\ref{equation-diagram}) shows that $\varphi^*E_X$ is the kernel
of a map
$\mathcal{O}_{\mathbb{P}^1}(d)^{\oplus 6} \to \mathcal{O}_{\mathbb{P}^1}$
given by substituting $(G_0, \ldots, G_5)$ into the partial derivatives
of the polynomial defining $X$. Since the equation is
$X_0^5 + \ldots + X_5^5$, the derivatives are $X_i^4$, and substituting
we obtain $G_i^4$ as desired.
\end{proof}

\section{Relating the Splitting Types}
\label{splittingrelation}
\noindent
Observe that $\Omega_X(\varphi)$ is also a graded free module of rank 5 and so has a splitting type, which we denote using $e_1,\dots,e_5$. In this section, we relate the splitting type of $\Omega_X(\varphi)$ to the splitting type of $E_X(\varphi)$. 

\medskip \noindent If $(A_0,\dots,A_5) \in \Omega_X(\varphi)$, then $A_0G_0 + \cdots + A_5G_5 = 0$ so that 
\[A_0^4G_0^4 + \cdots + A_5^4G_5^4 = 0\]
by the Frobenius endomorphism in characteristic 2. Let 
\[\mathcal{T} = \{(A_0^4,\dots,A_5^4) \mid (A_0,\dots,A_5) \in \Omega_X(\varphi)\} \]
in $E_X(\varphi)$. We denote the $R$-module generated by $\mathcal{T}$ as $R\langle\mathcal{T}\rangle$.

\begin{lemma} In the notation above, $E_X(\varphi) = R\langle\mathcal{T}\rangle$.
\end{lemma}

\begin{proof}
Let $(B_0,\dots,B_5)$ be an element of $E_X(\varphi)$ where $B_i$ is a homogeneous polynomial of degree $b$. We consider the case $b \equiv 0 \mod 4$.

\medskip \noindent Observe that we can rewrite each monomial term of $B_i$ as $(c^{1/4}S^{\ell}T^k)^4S^iT^{4-i}$ or $(c^{1/4}S^{\ell}T^k)^4$ for some integers $\ell,k$, where $c \in k$ and $0 < i <4$. After collecting terms and applying the Frobenius endomorphism, we obtain 
\[B_i = a_{i1}^4 + a_{i2}^4S^3T + a_{i3}^4S^2T^2 + a_{i4}^4ST^3\]
 where each $a_{ij}$ is an element of $R$.  Then, since $\displaystyle B_0G_0^4 + \cdots + B_5G_5^4 = 0$, substituting our expression for the $B_i$'s and applying Frobenius we obtain
\begin{eqnarray*}
(\sum_{i=0}^5 a_{i1}G_i)^4 + (\sum_{i=0}^5a_{i2}G_i)^4S^3T + (\sum_{i=0}^5a_{i3}G_i)^4S^2T^2 + (\sum_{i=0}^5a_{i4}G_i)^4ST^3  &=& 0
\end{eqnarray*}
The sums $\sum_{i=0}^5 a_{ij}G_i$ are each themselves homogeneous polynomials. But since the degree of $T$ in each term above is distinct modulo 4, the equation $\sum_{i=0}^5 a_{ij}G_i = 0$ implies that $(a_{0j},\dots,a_{5j}) \in \Omega_X(\varphi)$, so that $(a_{0j}^4,\dots,a_{5j}^4) \in \mathcal{T}$ for $1 \le j \le 4$.

\medskip\noindent
Hence, every homogeneous element of $E_X(\varphi)$ is contained in the submodule generated by $\mathcal{T}$. Since the reverse containment is trivial, it follows that $E_X(\varphi) = R\langle\mathcal{T}\rangle$. The cases for $b \equiv 1,2,3 \mod 4$ follow similarly. 
\end{proof}

\begin{proposition}
If $x_i = (x_{i0},\dots,x_{i5})$ for $1 \le i \le 5$ form a basis for $\Omega_X(\varphi)$, then $y_i = (x_{i0}^4,\dots,x_{i5}^4)$ for $1 \le i \le 5$ form a basis for $E_X(\varphi)$.
\end{proposition}

\begin{proof}
If $x_i \in \Omega_X(\varphi)$, then $y_i  \in \mathcal{T}$ and every element of $\mathcal{T}$ is an $R$-linear combination of the $y_i$'s. Since $E_X(\varphi) = R\langle\mathcal{T}\rangle$, every element of $E_X(\varphi)$ is also an $R$-linear combination of the $y_i$'s so that the $y_i$'s generate $E_X(\varphi)$. Moreover, $E_X(\varphi)$ is a free module of rank 5 over a domain, so the generators $y_i$ for $E_X(\varphi)$ must also be linearly independent and hence form a basis.  
\end{proof}
\noindent 

\noindent Accounting for twist, a simple computation using the results above gives us the following.

\begin{corollary}
\label{splittingformula}
If $f_i$ denotes the splitting type of $E_X(\varphi)$ and $e_i$ denotes the splitting type of $\Omega_X(\varphi)$, then for a degree $d$ morphism, $f_i = 4e_i + 5d$.
\end{corollary}

\section{Numerology}
\label{section-numerology}

\noindent We now utilize some facts about graded free modules in order to give constraints on potential splitting types.  Given a graded free module
\[ M = R(u_1) \oplus ... \oplus R(u_r) \]
one can observe that the Hilbert polynomial $H_M$ is given by
\[H_M(m) = rm + u_1 + ... + u_r + r. \]
Let $\varphi$ denote a free morphism into $X$.  Noting that the map $\tilde{\varphi}: R(-d)^{\oplus n+1}_m \longrightarrow R_m$ is surjective for $m \gg 0$, we obtain
\begin{align*}
H_{\Omega(\varphi)}(m) &= \dim_k\left(\ker(R(-d)^{\oplus n+1}_m \longrightarrow R_m)\right) \\
&= (n+1)(-d + m + 1) - (m+1) \\
&= nm + -d(n+1) + n.
\end{align*}
A similar calculation shows that,
\[H_{E_X(\varphi)}(m) = nm + d(n+1-5) + n\]
We continue to refer to the splitting type components of $\Omega(\varphi)$, respectively $E_X(\varphi)$ as $e_i$, respectively $f_i$.  In both cases $n=r=5$, so combining these two equations with the general form for the Hilbert polynomial of a graded free module, we obtain our first constraints:
\begin{align}
e_1 + e_2 + e_3 + e_4 + e_5 &= -6d \\
f_1 + f_2 + f_3 + f_4 + f_5 &= d.
\end{align}
Recall from Section \ref{section-setup} that a curve is free, respectively very free if $f_i \geq 0$, respectively $f_i > 0$ for each $i$.  Since $f_i = 4e_i + 5d$, it follows that
\begin{align}
e_i \geq -\frac{5d}{4}
\end{align}
where strict inequality implies the curve is very free.  With these two bounds, we can quickly observe a few facts about curves of different degrees.

\begin{remark} \ 
\label{charlie}
\begin{enumerate}
\item
There exist no free curves in degrees $1,2,3,6,$ and $7$.
\item
Any free curve of degree not divisible by $4$ must be very free.
\item
There are no very free curves in degrees $4$ or $8$.
\item
The $\Omega(\varphi)$ splitting type of a degree $4$ free curve must be $(-5,-5,-5,-5,-4)$.
\item
The $\Omega(\varphi)$ splitting type of a degree $5$ very free curve must be $(-6,-6,-6,-6,-6)$.
\end{enumerate} 
\end{remark}

\noindent All of these observation follow directly from the two constraints.  For example, in degree $6$, $e_1 + e_2 + e_3 +e_4 + e_5 = -6d = -36$.  However, each $e_i \geq \frac{-30}{4} = -7.5$.  So even if each $e_i$ is at best $-7$, the $e_i$ cannot sum to $-36$.  

\medskip \noindent The rest of the remarks follow in a similar manner.  Note that one can glean even more information about these curves from the constraints, but the remarks listed above are sufficient for our purposes.

\section{Degree 4 and 5 morphisms into X}
\label{section-degree-4-5}

\noindent
We will now show that there are no free morphisms of degrees $4$ or $5$ into $X$.
A morphism $\varphi = (G_0, ..., G_5)$, where each $G_i = \sum_{j=0}^d a_{ij}S^{d - j}T^{j}$ is a  homogeneous polynomials of degree $d$, gives us a $6 \times (d+1)$ matrix $(a_{ij})$. We will denote this matrix as $M_{\varphi}$. 

\begin{lemma}
\label{lemma-matrices-have-max-rank}
If $\varphi$ is a degree $4$ or $5$  free morphism into $X$, then $M_{\varphi}$ has maximal rank.
\end{lemma}

\begin{proof} This follows from Remark \ref{charlie} (4) and (5) by observing that for a degree $d$ morphism into $X$, the transpose of $M_{\varphi}$ is the matrix of the $k$-linear map $\tilde{\varphi}_d: (R(-d)^{\oplus 6})_d \rightarrow R_d$.
\end{proof}

\begin{lemma}
\label{lemma-morphisms-degree 4 and 5}
\ 
\begin{enumerate}
\item[(a)] There are no degree $4$ free morphisms into $X$.
\item[(b)] There are no degree $5$ free morphisms into $X$.
\end{enumerate}
\end{lemma}

\begin{proof} (a) Assume a degree $4$ free morphism $\varphi = (G_0, ..., G_5)$ exists. By the previous lemma, the $6 \times 5$ matrix $M_{\varphi} = (a_{ij})$ has maximal rank. Since permuting the $G_i's$ does not affect the splitting type of $E_X(\varphi)$, we can assume that the first $5$ rows of $M_{\varphi}$ are linearly independent over $k$. Then det$((a_{ij})_{i \leq 4}) \neq 0$.
Now consider the matrix $\overline{M_{\varphi}} = (a^4_{ij})$. By the Frobenius endomorphism on $k$, $\det((a^4_{ij})_{i \leq 4})$ = $\det((a_{ij})_{i \leq 4}))^4 \neq 0$, proving that $\overline{M_{\varphi}}$ has maximal rank as well.

\medskip \noindent
Since $G^5_0 + ... + G^5_5 = 0$, computing the coefficients of $G^5_0 + ... + G^5_5 $, we obtain for $0 \leq j \leq 4$
\begin{equation}
\sum\nolimits_{i=0}^5 a^4_{ij}a_{i1} = 0
\quad\text{and}\quad
\sum\nolimits_{i=0}^5 a^4_{ij}a_{i3} = 0.
\end{equation}
The kernel of the map $k^6 \rightarrow k^5$ given by right multiplication by the matrix $\overline{M_{\varphi}}$ has dimension $1$ because rank$\big{(}\overline{M}_{\varphi}\big{)} = 5$. By (5.2.1), $(a_{01}, a_{11}, ...,a_{51})$, $(a_{03}, a_{13}, ..., a_{53})$ $\in$ ker$(k^6 \rightarrow k^5)$, and since  these $6$-tuples are columns of $M_{\varphi}$, they are linearly independent over $k$. Then dim$_k$\big{(}ker$(k^6 \rightarrow k^5)\big{)}$ $\geq$ $2$, a contradiction.\\

\noindent
(b) Assume $\varphi = (G_0, ..., G_5)$ is a degree $5$ free morphism. By the previous lemma, the matrix $M_{\varphi} = (a_{ij})$ has maximal rank, and is invertible. Thus $\overline{M_{\varphi}} = (a^4_{ij})$ is invertible by the same argument above.
Since $G^5_0 + ... + G^5_5 = 0$, computing the coefficients of the polynomial $G^5_0 + \cdots + G^5_5$, we get
$\sum_{i=0}^5 a^4_{ij}a_{i2} = 0$ for $0 \leq j \leq 5$. 
Thus, the product of the row matrix $(a_{02}, a_{12}, ... , a_{52})$ and the matrix $\overline{{M}_{\varphi}}$ is $0$, which is impossible because $(a_{02}, a_{12}, ... , a_{52}) \neq 0$ and $\overline{{M}_{\varphi}}$ is invertible.
\end{proof}

\section{Computations for the degree $8$ free curve}
\label{section-degree-8}

\noindent Let $\varphi:\mathbb{P}^1\to\mathbb{P}^5$ be a morphism given by the 6-tuple
\begin{align*}
&G_0 = S^7T \\
&G_1 = S^4T^4 + S^3T^5 \\
&G_2 = S^4T^4 + S^3T^5 + T^8 \\
&G_3 = S^7T + S^6T^2 + S^5T^3 + S^4T^4 + S^3T^5 \\
&G_4 = S^8 + S^7T + S^6T^2 + S^5T^3 + S^4T^4 + S^3T^5 + T^8 \\
&G_5 = S^8 + S^7T + S^6T^2 + S^5T^3 + S^4T^4 + S^3T^5 + S^2T^6 + ST^7.
\end{align*}

\noindent One can check by computer or by hand that this curve lies on the Fermat hypersurface $X\subset\mathbb{P}^5$. Due to twisting, the domain of the map $\tilde{\varphi}:R(-8)^{\oplus 6}\to R$ has its first nontrivial graded piece in dimension $8$. The $G_i$ are linearly independent over $k$, hence the kernel is trivial in dimension $8$. The matrix for the map $\tilde{\varphi}_9:R(-8)_{9}^{\oplus 6}\to R_9$ is
$$
\begin{pmatrix}
0 &0 &0 &0 &0 &0 &0 &0 &1 &0 &1 &0 \\
1 &0 &0 &0 &0 &0 &1 &0 &1 &1 &1 &1 \\
0 &1 &0 &0 &0 &0 &1 &1 &1 &1 &1 &1 \\
0 &0 &0 &0 &0 &0 &1 &1 &1 &1 &1 &1 \\
0 &0 &1 &0 &1 &0 &1 &1 &1 &1 &1 &1 \\
0 &0 &1 &1 &1 &1 &1 &1 &1 &1 &1 &1 \\
0 &0 &0 &1 &0 &1 &0 &1 &0 &1 &1 &1 \\
0 &0 &0 &0 &0 &0 &0 &0 &0 &0 &1 &1 \\
0 &0 &0 &0 &1 &0 &0 &0 &1 &0 &0 &1 \\
0 &0 &0 &0 &0 &1 &0 &0 &0 &1 &0 &0
\end{pmatrix}
$$

\noindent where each direct summand of the domain has a basis $\left\{(S,0),(0,T)\right\}$, of which we take six copies (for total dimension $12$), and the range has basis given by the degree $9$ monomials in $S$ and $T$, ordered by increasing $T$-degree (for total dimension $10$). This matrix has rank $10$, which means that the map in degree $9$ is surjective. By rank-nullity, two dimensions of the kernel live in degree $9$; denote the generators by $x_1,x_2$. Surjectivity of $\tilde{\varphi}$ in degree $9$ implies surjectivity in all higher degrees. A second application of rank-nullity gives $\dim_k\Omega(\varphi)_{10}=7$. Four of the generators are inherited from the previous degree, taking the forms
$$
x_1S,x_2S,x_1T,x_2T.
$$

\noindent We conclude that there are three additional generators in degree $10$. Therefore, the splitting type of $\Omega_X(\varphi)$ is $(e_1,...,e_5)=(-10,-10,-10,-9,-9)$, which corresponds to a splitting type for $E_X(\varphi)$ of $(f_1,...,f_5)=(0,0,0,4,4)$, hence the curve is free.

\section{A Very Free Rational Curve of Degree 9}
\label{section-degree-9}

\noindent
We conclude by giving an example of a degree 9 very free curve lying on $X$. Let $\varphi:\mathbb{P}^1 \rightarrow \mathbb{P}^5$ be a morphism into the Fermat hypersurface given by the 6-tuple
\begin{align*}
G_0 &= S^4T^5 \\
G_1 &= S^9 + S^8T + S^5T^4 \\
G_2 &= S^9 + S^4T^5 + ST^8 \\  
G_3 &= S^9 + S^8T + S^4T^5 + S^3T^6 + S^2T^7 + ST^8 \\ 
G_4 &= S^9 + S^5T^4 + S^3T^6 +S^2T^7 + ST^8 + T^9 \\ 
G_5 &= S^7T^2 + S^6T^3 + S^5T^4 + S^3T^6 + S^2T^7 +ST^8 +T^9.
\end{align*}
\noindent Let $e_1,...,e_5$ again denote the splitting type of $\Omega_X (\varphi)$.  As in Section \ref{section-degree-8}, we know that $e_i \le -9$.  Since the $G_i$ are linearly independent over $k$, $\dim_k(\Omega_X(\varphi)_9) = 0$.  Next we claim that $\tilde{\varphi_{10}} : R_1^{\bigoplus6} \rightarrow R_{10}$ is surjective.  In fact, it can be checked that the $\tilde{\varphi}(b_i)$ span $R_{10}$, where the $b_i$ are distinct basis elements of $R_1^{\bigoplus6}$.  It follows that $\tilde{\varphi_n} : R(-9)_n^{\bigoplus6} \rightarrow R_n$ is surjective for $n \ge 10$.  Hence, 
\begin{align*}
\dim_k(\Omega_X(\varphi)_{10}) &= \dim_k(R_1^{\bigoplus6}) - \dim_k(R_{10}) = 1\\
\dim_k(\Omega_X(\varphi)_{11}) &= \dim_k(R_2^{\bigoplus6}) - \dim_k(R_{11}) = 6.
\end{align*}

\noindent
After reordering, this yields $(e_1,...,e_5) = (-11,-11,-11,-11,-10)$, which corresponds to the splitting type $(1,1,1,1,5)$ of $E_X(\varphi)$, showing that $\varphi$ is very free.  This completes the proof of Theorem \ref{theorem-main}.

\bibliography{bibli}
\bibliographystyle{gost2008}
\end{document}